\documentclass[a4paper]{amsart}

\usepackage{amsmath,amssymb,amsthm}
\usepackage[all]{xy}
\usepackage{eucal}
\usepackage[dvipdfmx]{graphicx}
\usepackage{tikz}
\usepackage[dvipdfmx,colorlinks=true,backref=page]{hyperref}

\newtheorem{theorem}{Theorem}[section]
\newtheorem{proposition}[theorem]{Proposition}
\newtheorem{lemma}[theorem]{Lemma}

\theoremstyle{definition}

\theoremstyle{remark}

\numberwithin{equation}{section}

\newcommand{\hocolim}{\operatorname{hocolim}}
\newcommand{\rt}{\operatorname{\widetilde{\rtimes}}}

\SelectTips{cm}{}

\title[Beben and Theriault's theorem on homotopy fibers]{A short elementary proof of Beben and Theriault's theorem on homotopy fibers}

\author[D. Kishimoto]{Daisuke Kishimoto}
\address{Faculty of Mathematics, Kyushu University, Fukuoka, 819-0395, Japan}
\email{kishimoto@math.kyushu-u.ac.jp}

\author[Y. Minowa]{Yuki Minowa}
\address{Department of Mathematics, Kyoto University, Kyoto, 606-8502, Japan}
\email{minowa.yuki.48z@st.kyoto-u.ac.jp}

\date{\today}



\begin{document}

\maketitle

\begin{abstract}
  Beben and Theriault proved a theorem on the homotopy fiber of an extension of a map with respect to a cone attachment, which has produced several applications. We give a short and elementary proof of this theorem.
\end{abstract}


\section{Introduction}\label{Introduction}

It is a fundamental problem in algebraic topology to describe how the homotopy type of a space changes after a cone attachment, and the problem has been intensely studied in connection with LS category. Describing the effect of a cone attachment on the homotopy type of a related space is of particular importance too. For example, relations between $\Omega X$ and $\Omega(X\cup CA)$ were studied in \cite{FT,G,HL}. Beben and Theriault \cite{BT} considered a homotopy commutative diagram
\begin{equation}
  \label{diagram}
  \xymatrix{
    &F\ar[r]^h\ar[d]^j&F'\ar[d]^{j'}\\
    A\ar[r]^f&E\ar[r]\ar[d]^p&E'\ar[d]^{p'}\\
    &B\ar@{=}[r]&B
  }
\end{equation}
where $B$ is path-connected, the middle row is a homotopy cofibration, and the two columns are homotopy fibrations. They gave a nice description of $F'$ in terms of $F$. To state this result, we set notation. The half smash product of spaces $X$ and $Y$ is defined by
\[
  X\rtimes Y=(X\times Y)/(*\times Y).
\]
Let $\epsilon\colon\Sigma\Omega X\to X$ denote the evaluation map.

\begin{theorem}
  \label{main}
  Consider the homotopy commutative diagram \eqref{diagram}. If the map $\Omega p\colon\Omega E\to\Omega B$ has a right homotopy inverse $s\colon\Omega B\to\Omega E$, then there is a homotopy cofibration
  \[
    A\rtimes\Omega B\xrightarrow{\theta}F\xrightarrow{h}F'.
  \]
  Moreover, if $A$ is a suspension, then there is a homotopy commutative diagram
  \[
    \xymatrix{
      A\rtimes\Omega B\ar[rr]^(.57)\theta\ar[d]_\simeq&&F\ar[d]^j\\
      (A\wedge\Omega B)\vee A\ar[rr]^(.64){[f,\epsilon\circ\Sigma s]+f}&&E.
    }
  \]
\end{theorem}

The homotopy commutative diagram \eqref{diagram} appears in several contexts, and so Theorem \ref{main} has been applied to produce several interesting results such as a loop space decomposition of certain manifolds \cite{C1,C2,HT1,HT2,T}. However, Beben and Theriault's proof of Theorem \ref{main} is long and complicated, which needs a delicate analisys of the action of the loop space on the fiber of a fibration and to consider relative Whitehad products. In this paper, we provide a short and elementary proof of Theorem \ref{main}. Our proof is basically an elementary analysis of (homotopy) pushouts, and does not need a delicate analysis of the action in a fibration and relative Whitehead products.

We will always assume that every space has a non-degenerate basepoint and every map is basepoint preserving.

\subsection*{Acknowledgement}

The first author was partially supported by JSPS KAKENHI JP22K03284 and JP19K03473, and the second author was partially supported by JST, the establishment of university fellowships towards the creation of science technology innovation, Grant Number JPMJFS2123.


\section{Proof}\label{Proof}

We consider the commutative diagram \eqref{diagram}, and assume that the map $\Omega p\colon\Omega E\to\Omega B$ has a right homotopy inverse. Let
\[
  \mathsf{E}=\{(x,l)\in E\times B^{[0,1]}\mid p(x)=l(0)\}
\]
and define a map $\mathsf{p}\colon\mathsf{E}\to B$ by $\mathsf{p}(x,l)=l(1)$. Let
\[
  \mathsf{F}=\{(x,l)\in E\times B^{[0,1]}\mid p(x)=l(0),\,l(1)=*\}.
\]
Then we get a fibration sequence
\begin{equation}
  \label{fibration}
  \mathsf{F}\xrightarrow{\mathsf{j}}\mathsf{E}\xrightarrow{\mathsf{p}}B
\end{equation}
and a homotopy action
\[
  \Gamma\colon\mathsf{F}\times\Omega B\to\mathsf{F},\quad((x,l),l')\mapsto (x,l*l').
\]
Clearly, we may replace the homotopy fibration $F\xrightarrow{j}E\xrightarrow{p}B$ in \eqref{diagram} with the fibration sequence \eqref{fibration}. In particular, the map $f\colon A\to E$ will be replaced with the composite $\mathsf{f}\colon A\xrightarrow{f}E\xrightarrow{\rm incl}\mathsf{E}$. Since $\mathsf{p}\circ\mathsf{f}\simeq*$ and $\mathsf{p}\colon\mathsf{E}\to B$ is a fibration, there is a map $\mathsf{f}'\colon A\to\mathsf{E}$ such that $\mathsf{f}\simeq\mathsf{f}'$ and $\mathsf{p}\circ\mathsf{f}'=*$. Then we may assume $\mathsf{p}\circ\mathsf{f}=*$. Let $p_i$ denote the $i$-th projection.

Let $\mathrm{Cyl}(g)$ denote the mapping cylinder of a map $g\colon X\to Y$. Let $i\colon X\to\mathrm{Cyl}(g)$ and $q\colon I_g\to Y$ denote the inclusion and the projection, respectively.

\begin{lemma}
  \label{Gamma}
  There is a commutative diagram
  \[
    \xymatrix{
    \mathsf{F}\times\Omega B\ar[r]^i\ar[d]^{p_1}&\mathrm{Cyl}(\Gamma)\ar[d]^{\mathsf{j}'}\\
    \mathsf{F}\ar[r]^{\mathsf{j}}&\mathsf{E}.
    }
  \]
\end{lemma}

\begin{proof}
  By the definition of the map $\Gamma$, there is a homotopy commutative diagram
  \[
    \xymatrix{
    \mathsf{F}\times\Omega B\ar[r]^(.6)\Gamma\ar[d]^{p_1}&\mathsf{F}\ar[d]^{\mathsf{j}}\\
    \mathsf{F}\ar[r]^{\mathsf{j}}&\mathsf{E}.
  }
  \]
  Then the statement is proved by the usual homotopy extension argument.
\end{proof}

Since we are assuming that the map $\Omega p\colon\Omega E\to\Omega B$ has a right homotopy inverse,  the map $\Omega\mathsf{p}\colon\Omega\mathsf{E}\to\Omega B$ has a right homotopy inverse too, say $\mathsf{s}\colon\Omega B\to\Omega\mathsf{E}$. In particular, the map $\Omega B\to\mathsf{F}$ is null-homotopic, and we fix a null homotopy $H\colon C\Omega B\to\mathsf{F}$. Let $\bar{\epsilon}\colon C\Omega X\to X$ denote the composite $C\Omega X\xrightarrow{\rm proj}\Sigma\Omega X\xrightarrow{\epsilon}X$.

\begin{lemma}
  There is a commutative diagram
  \[
    \xymatrix{
      \mathsf{F}\vee C\Omega\mathsf{E}\ar[rr]^{i+i\circ H\circ C\Omega\mathsf{p}}\ar[d]^i&&\mathrm{Cyl}(\Gamma)\ar[d]^{\mathsf{j}'}\\
      \mathrm{Cyl}(\mathsf{j}+\bar{\epsilon})\ar[rr]^{q'}&&\mathsf{E}
    }
  \]
  where the map $q'$ is homotopic to the projection $q\colon\mathrm{Cyl}(\mathsf{j}+\bar{\epsilon})\to\mathsf{E}$.
\end{lemma}

\begin{proof}
  There is a homotopy commutative diagram
  \[
    \xymatrix{
      \mathsf{F}\vee C\Omega\mathsf{E}\ar[rr]^{i+i\circ H\circ C\Omega\mathsf{p}}\ar[d]^{\mathsf{j}+\bar{\epsilon}}&&\mathrm{Cyl}(\Gamma)\ar[d]^{\mathsf{j}'}\\
      \mathsf{E}\ar@{=}[rr]&&\mathsf{E}.
    }
  \]
  Then the usual homotopy extension argument proves the statement.
\end{proof}

Let $X\rt Y=(X\times Y)\cup(*\times CY)$. We define a map $\rho\colon X\rt\Omega Y\to X\vee Y$ by the induced map between the pushouts of the two rows in the commutative diagram
\[
  \xymatrix{
    X\times\Omega Y\ar[d]^{p_1}&\Omega Y\ar[l]\ar[r]\ar[d]&C\Omega Y\ar[d]^{\bar{\epsilon}}\\
    X&\ast\ar[l]\ar[r]&Y.
  }
\]
Then $\rho$ is natural with respect to $X$ and $Y$.

\begin{lemma}
  \label{Gamma-rho}
  There is a homotopy commutative diagram
  \[
    \xymatrix{
      \mathsf{F}\rt\Omega\mathsf{E}\ar[rr]^{\widetilde{\Gamma}\circ(1\rt\Omega\mathsf{p})}\ar[d]^\rho&&\mathsf{F}\ar[d]^{\mathsf{j}}\\
      \mathsf{F}\vee\mathsf{E}\ar[rr]^{\mathsf{j}+1}&&\mathsf{E}
    }
  \]
  where $\widetilde{\Gamma}$ is an extension of $\Gamma$.
\end{lemma}

\begin{proof}
  Consider a diagram
  \[
    \xymatrix@!C=18pt{
      &C\Omega\mathsf{E}\ar[ld]\ar[dd]|\hole^(.3){i\vert_{C\Omega\mathsf{E}}}&&\Omega\mathsf{E}\ar[rr]\ar[ll]\ar[ld]\ar[dd]|{\hole}&&\mathsf{F}\times\Omega\mathsf{E}\ar[dl]_i\ar[dd]^(.4){p_1}\\
      \mathsf{F}\vee C\Omega\mathsf{E}\ar[dd]^(.4)i&&\mathsf{F}\vee C\Omega\mathsf{E}\ar[rr]^(.56){i+i\circ H\circ C\Omega\mathsf{p}}\ar@{=}[ll]\ar[dd]^(.4)i&&\mathrm{Cyl}(\Gamma)\ar[dd]^(.4){\mathsf{j}'}\\
      &\mathrm{Cyl}(\mathsf{j}+\bar{\epsilon})\ar@{=}[dl]&&\ast\ar[dl]\ar[ll]|{\hole}\ar[rr]|{\hole}&&\mathsf{F}\ar[dl]_{\mathsf{j}}\\
      \mathrm{Cyl}(\mathsf{j}+\bar{\epsilon})&&\mathrm{Cyl}(\mathsf{j}+\bar{\epsilon})\ar[rr]^{q'}\ar@{=}[ll]&&\mathsf{E}.
    }
  \]
  The right top face and the right front face is commutative by the above construction, and the right face is commutative by Lemma \ref{Gamma}. Clearly, other faces are commutative, and so the above diagram is commutative. Then by taking the pushouts of all rows, we get a commutative diagram
  \[
    \xymatrix{
      \mathsf{F}\rt\Omega\mathsf{E}\ar[r]\ar[d]&\mathrm{Cyl}(\Gamma)\ar[d]^{\mathsf{j}'}\\
      \mathsf{F}\vee \mathrm{Cyl}(\mathsf{j}+\bar{\epsilon})\ar[r]^(.7){\mathsf{j}+q'}&\mathsf{E}.
    }
  \]
  By construction, this commutative diagram extends to a homotopy commutative diagram
  \[
    \xymatrix{
      \mathsf{F}\rt\Omega\mathsf{E}\ar@/^20pt/[rr]^{\widetilde{\Gamma}\circ(1\rt\Omega\mathsf{p})}\ar[r]\ar[d]\ar@/_40pt/[dd]_\rho&\mathrm{Cyl}(\Gamma)\ar[d]^{\mathsf{j}'}\ar[r]^q&\mathsf{F}\ar[d]^{\mathsf{j}}\\
      \mathsf{F}\vee \mathrm{Cyl}(\mathsf{j}+\bar{\epsilon})\ar[r]^(.7){\mathsf{j}+q'}\ar[d]^{1\vee q}&\mathsf{E}\ar@{=}[r]\ar@{=}[d]&\mathsf{E}\\
      \mathsf{F}\vee\mathsf{E}\ar[r]^(.6){\mathsf{j}+1}&\mathsf{E}.
  }
  \]
  Thus the proof is finished.
\end{proof}

To identify the fiber $F'$ in \eqref{diagram}, we will use:

\begin{lemma}
  [{\cite[Appendix HL, Proposition]{DF}}]
  \label{fiber}
  Let $\{F_i\to E_i\to B\}_{i\in I}$ be an $I$-diagram of homotopy fibrations over a common path-connected base $B$. Then the sequence
  \[
    \underset{I}{\hocolim}\,F_i\to\underset{I}{\hocolim}\,E_i\to B
  \]
  is a homotopy fibration.
\end{lemma}

We recall a well known property of homotopy pushouts. See \cite[Proposition 6.2.6]{A} for a proof.

\begin{lemma}
  \label{hpo}
  Let $W$ denote the homotopy pushout of a cotriad
  \[
    X\xleftarrow{g}Y\to Z.
  \]
  If $g$ is a cofibration, then the projection $W\to X\cup_YZ$ is a homotopy equivalence.
\end{lemma}

Let $\tilde{\mathsf{f}}\colon A\to\mathsf{F}$ be the lift of the map $\mathsf{f}$.

\begin{proposition}
  \label{cofibration fiber}
  There is a homotopy cofibration
  \[
    A\rt\Omega B\xrightarrow{\bar{\theta}}\mathsf{F}\to F'.
  \]
\end{proposition}

\begin{proof}
  By Lemma \ref{Gamma}, there is a commutative diagram
  \[
    \xymatrix{
      CA\times\Omega B\ar[d]^{p_1}&A\times\Omega B\ar[rr]^(.6){i\circ(\tilde{\mathsf{f}}\times 1)}\ar[l]\ar[d]^{p_1}&&\mathrm{Cyl}(\Gamma)\ar[d]^{\mathsf{j}'}\\
      CA\ar[d]^\ast&A\ar[rr]^{\mathsf{f}}\ar[l]\ar[d]^{\mathsf{p}\circ\mathsf{f}}&&\mathsf{E}\ar[d]^{\mathsf{p}}\\
      B\ar@{=}[r]&B\ar@{=}[rr]&&B.
    }
  \]
  where all columns are homotopy fibrations. Let $\mathsf{F}'$ be the oushout of the top row. Then by Lemmas \ref{fiber} and \ref{hpo}, we get a homotopy fibration $\mathsf{F}'\to\mathsf{E}\cup CA\to B$, hence a homotopy equivalence
  \[
    F'\simeq\mathsf{F}'.
  \]
  Since the map $\Omega\mathsf{p}$ has a right homotopy inverse, the restriction of $i\colon\mathsf{F}\times\Omega B\to\mathrm{Cyl}(\Gamma)$ to $*\times\Omega B$ is null-homotopic. Then we get a commutative diagram
  \[
    \xymatrix{
      \Omega B\ar[d]&\Omega B\ar[rr]\ar[l]\ar[d]&&\ast\ar[d]\\
      CA\times\Omega B\ar[d]&A\times\Omega B\ar[rr]^(.6){i\circ(\tilde{\mathsf{f}}\times 1)}\ar[l]\ar[d]&&\mathrm{Cyl}(\Gamma)\ar[d]^q\\
      CA\rt\Omega B&A\rt\Omega B\ar[rr]^(.6){\widetilde{\Gamma}\circ(\tilde{\mathsf{f}}\rt 1)}\ar[l]&&\mathsf{F}
    }
  \]
  where all columns are homotopy cofibrations and $\widetilde{\Gamma}\colon\mathsf{F}\rt\Omega B\to\mathsf{F}$ is an extension of $\Gamma$. Let $\mathsf{F}''$ be the pushout of the bottom row. Then since homotopy pushouts commute with taking cofibers, we get a homotopy cofibration $*\to\mathsf{F}'\to\mathsf{F}''$ by Lemma \ref{hpo}, hence a homotopy equivalence
  \[
    \mathsf{F}'\simeq\mathsf{F}''.
  \]
  On the other hand, since $CA\rt\Omega B$ is contractible, we have a homotopy cofibration
  \[
    A\rt\Omega B\xrightarrow{\widetilde{\Gamma}\circ(\tilde{\mathsf{f}}\rt 1)}\mathsf{F}\to\mathsf{F}''.
  \]
  Thus the proof is finished.
\end{proof}

\begin{lemma}
  \label{rho}
  There is a natural map $w\colon\Omega X*\Omega Y\to X\rt\Omega Y$ satisfying the commutative diagram
  \[
    \xymatrix{
      \Omega X*\Omega Y\ar@{=}[r]\ar[d]^w&\Omega X*\Omega Y\ar[d]^{[\epsilon,\epsilon]}\\
      X\rt\Omega Y\ar[r]^\rho&X\vee Y.
    }
  \]
\end{lemma}

\begin{proof}
  Consider a commutative diagram
  \[
    \xymatrix{
      C\Omega X\times\Omega Y\ar[d]^{\tilde{\epsilon}\times 1}&\Omega X\times\Omega Y\ar[l]\ar[r]\ar[d]^{p_2}&\Omega X\times C\Omega Y\ar[d]^{p_2}\\
      X\times\Omega Y\ar[d]^{p_2}&\Omega Y\ar[l]\ar[r]\ar[d]&C\Omega Y\ar[d]^{\tilde{\epsilon}}\\
      X&\ast\ar[l]\ar[r]&Y.
    }
  \]
  Then by taking the pushouts of the three rows, we obtain the diagram in the statement.
\end{proof}

Let $i\colon X\to X\rt Y$ denote the inclusion, and let $\overline{w}$ denote the composite
\[
  X*\Omega Y\xrightarrow{E*1}\Omega\Sigma X*\Omega Y\xrightarrow{w}\Sigma X\rt\Omega Y.
\]
where $E\colon Y\to\Omega\Sigma Y$ is the suspension map. By definition, $\overline{w}$ is natural with respect to $X$ and $Y$.

\begin{lemma}
  \label{decomposition}
  The map
  \[
    \overline{w}+i\colon(X*\Omega Y)\vee\Sigma X\to\Sigma X\rt\Omega Y
  \]
  is a homotopy equivalence.
\end{lemma}

\begin{proof}
  There is a commutative diagram
  \[
    \xymatrix{
      X\times C\Omega Y\ar[d]^{p_2}&X\times\Omega Y\ar[l]\ar[r]\ar[d]^{p_2}&C X\times \Omega Y\ar[d]^{\tilde{\epsilon}\circ CE\times 1}\\
      C\Omega Y\ar[d]&\Omega Y\ar[r]\ar[l]\ar[d]&\Sigma X\times\Omega Y\ar[d]\\
      \Sigma X\rt C\Omega Y&\Sigma X\rt \Omega Y\ar@{=}[r]\ar[l]&\Sigma X\rt \Omega Y
    }
  \]
  where all columns are homotopy cofibrations. Then since the inclusion $i\colon\Sigma X\to\Sigma X\rt C\Omega Y$ is a homotopy equivalence, we get a homotopy cofibration
  \[
    X*\Omega Y\xrightarrow{\overline{w}}\Sigma X\rt\Omega Y\xrightarrow{q}\Sigma X
  \]
  by taking the pushouts of the three rows. By definition, we have $q\circ i\simeq 1$, and so there is a homotopy commutative diagram
  \[
    \xymatrix{
      X*\Omega Y\ar[r]^(.39){\rm incl}\ar@{=}[d]&(X*\Omega Y)\vee\Sigma X\ar[r]^(.68){\rm proj}\ar[d]^{\overline{w}+i}&\Sigma X\ar@{=}[d]\\
      X*\Omega Y\ar[r]^(.47){\overline{w}}&\Sigma X\rt\Omega Y\ar[r]^(.6)q&\Sigma X
    }
  \]
  where the two rows are homotopy cofibrations. Then the middle vertical map is a homotopy equivalence, completing the proof.
\end{proof}

Now we are ready to prove Theorem \ref{main}.

\begin{proof}
  [Proof of Theorem \ref{main}]
  Since the projection $A\rt\Omega B\to A\rtimes\Omega B$ is a homotopy equivalence, the first statement follows from Proposition \ref{cofibration fiber}. Suppose $A=\Sigma\overline{A}$. By Lemmas \ref{Gamma-rho} and \ref{rho}, there is a homotopy commutative diagram
  \begin{equation}
    \label{big diagram}
    \xymatrix@!C{
      \overline{A}*\Omega B\ar[r]^{\overline{w}}\ar[d]^{1*E}&A\rt\Omega B\ar[r]^{1\rt\mathsf{s}}\ar[d]^{1\rt E}&A\rt\Omega\mathsf{E}\ar[r]^{\tilde{\mathsf{f}}\rt 1}\ar[d]^{1\rt E}&\mathsf{F}\rt\Omega\mathsf{E}\ar@{=}[d]\\
      \overline{A}*\Omega\Sigma\Omega B\ar[r]^{\overline{w}}\ar[d]^{E*1}&A\rt\Omega\Sigma\Omega B\ar[r]^{1\rt\Omega\Sigma\mathsf{s}}\ar[d]^\rho&A\rt\Omega\Sigma\Omega\mathsf{E}\ar[r]^(.55){\tilde{\mathsf{f}}\rt\Omega\epsilon}\ar[d]^\rho&\mathsf{F}\rt\Omega\mathsf{E}\ar[r]^(.59){\widetilde{\Gamma}\circ(1\rt\mathsf{p})}\ar[d]^\rho&\mathsf{F}\ar[d]^{\mathsf{j}}\\
      \Omega A*\Omega\Sigma\Omega B\ar[r]^(.55){[\epsilon,\epsilon]}&A\vee \Sigma\Omega B\ar[r]^{1\vee\Sigma\mathsf{s}}&A\vee\Sigma\Omega\widehat{E}\ar[r]^(.55){\tilde{\mathsf{f}}\vee\epsilon}&\mathsf{F}\vee\mathsf{E}\ar[r]^(.55){\mathsf{j}+1}&\mathsf{E}.
    }
  \end{equation}
  By definition, the map $\hat{\theta}\colon A\rt\Omega B\to\widehat{F}$ in Proposition \ref{fiber} is homotopic to the composite
  \[
    A\rt\Omega B\xrightarrow{1\rt\mathsf{s}}A\rt\Omega\mathsf{E}\xrightarrow{\tilde{\mathsf{f}}\rt 1}\mathsf{F}\rt\Omega\mathsf{E}\xrightarrow{\widetilde{\Gamma}\circ(1\rt\mathsf{p})}\mathsf{F}.
  \]
  Then the composite $\overline{A}*\Omega B\xrightarrow{\overline{w}}A\rt\Omega B\xrightarrow{\bar{\theta}}\mathsf{F}\xrightarrow{\mathsf{j}}\mathsf{E}$
  is homotopic to the left-bottom perimeter of \eqref{big diagram} which is the Whitehead product $[\mathsf{f},\epsilon\circ\Sigma\mathsf{s}]$. By the definition of $\hat{\theta}$, the composite
  \[
    A\xrightarrow{i}A\rt\Omega B\xrightarrow{\bar{\theta}}\mathsf{F}\xrightarrow{\mathsf{j}}\mathsf{E}
  \]
  equals $\mathsf{f}$. Thus by Lemma \ref{decomposition}, the second statement is proved, and therefore the proof is finished by applying the natural homotopy equivalences $\overline{A}*\Omega B\simeq A\wedge\Omega B$, $A\rt\Omega B\simeq A\rtimes\Omega B$ and $\mathsf{F}\simeq F$.
\end{proof}

\end{document}